\documentclass[12pt]{article}
\usepackage[a4paper,margin=1.0in]{geometry}
\usepackage[colorlinks,citecolor=magenta,linkcolor=black]{hyperref}
\pdfpagewidth=\paperwidth \pdfpageheight=\paperheight
\usepackage{amsfonts,amssymb,amsthm,amsmath,eucal,tabu,url}
\usepackage{pgf}
 \usepackage{array}
 \usepackage{tikz-cd}
 \usepackage{pstricks}
 \usepackage{pstricks-add}
 \usepackage{pgf,tikz}
 \usetikzlibrary{automata}
 \usetikzlibrary{arrows}
 \usepackage{indentfirst}
\usepackage{tabularx} 

\theoremstyle{plain}
\newtheorem{thm}{Theorem}[section]
\newtheorem{theorem}[thm]{Theorem}

\theoremstyle{definition}
\newtheorem{definition}[thm]{Definition}
\newtheorem{remark}[thm]{Remark}

\newtheorem{problem}[thm]{Problem}

\newtheorem{thevarthm}[thm]{\varthmname}

\newenvironment{varthm*}[1]{\trivlist\item[]{\bf #1.}\it}{\endtrivlist}

\renewcommand\geq{\geqslant}

\renewcommand\leq{\leqslant}

\newcommand\be{\begin{eqnarray*}}
\newcommand\ee{\end{eqnarray*}}

\newcommand\newop[2]{\def#1{\mathop{\rm #2}\nolimits}}
\newop\log{log}
\newop\ord{ord}
\newop\Gal{Gal}
\newop\SL{SL}
\newop\Bl{Bl}
\newop\mult{mult}
\newop\mass{mass}
\newop\div{div}
\newop\codim{codim}
\newop\sing{sing}
\newop\vdim{vdim}
\newop\edim{edim}
\newop\Ass{Ass}
\newop\size{size}
\newop\reg{reg}
\newop\satdeg{satdeg}
\newop\supp{supp}
\newop\Neg{Neg}
\newop\Nef{Nef}
\newop\Nefh{Nef_H}
\newop\Eff{Eff}
\newop\Zar{Zar}
\newop\MB{MB}
\newop\MBxC{MB\mathit{(x,C)}}
\newop\NnB{NnB}
\newop\Bigg{Big}
\newop\Effbar{\overline{\Eff}}

\theoremstyle{plain}
\newtheorem{custom}{{\rm Theorem}}

\def\keywordname{{\bfseries Keywords}}%
\def\keywords#1{\par\addvspace\medskipamount{\rightskip=0pt plus1cm
\def\and{\ifhmode\unskip\nobreak\fi\ $\cdot$
}\noindent\keywordname\enspace\ignorespaces#1\par}}
\def\subclassname{{\bfseries Mathematics Subject Classification
(2020)}\enspace}
\def\subclass#1{\par\addvspace\medskipamount{\rightskip=0pt plus1cm
\def\and{\ifhmode\unskip\nobreak\fi\ $\cdot$
}\noindent\subclassname\ignorespaces#1\par}}

\begin{document}
\title{On cubic-line arrangements with simple singularities}
\author{Przemys\l aw Talar}
\date{\today}
\maketitle

\thispagestyle{empty}
\begin{abstract}
In the present note we study combinatorial and algebraic properties of cubic-line arrangements in the complex projective plane admitting nodes, ordinary triple and $A_{5}$ singular points. We deliver a Hirzebruch-type inequality for such arrangement and we study the freeness of such arrangements providing an almost complete classification.
\subclass{14N20, 14C20}
\end{abstract}
\section{Introduction}
The main aim of the present note is to start a systematic study on arrangements of plane curves consisting of smooth cubic curves and lines admitting some simple singularities. Our motivation comes from the recent papers devoted to arrangements consisting of conics and lines in the plane, or generally speaking, arrangements of rational curves in the plane, see for instance \cite{DimcaPokora, PokSz}. From that perspective, it is worth recalling a recent paper by Dimca and Pokora \cite{DimcaPokora} devoted to conic-line arrangements admitting nodes, tacnodes, and ordinary triple points as singularities. In the aforementioned paper, the authors deliver a Hirzebruch-type inequality for weak combinatorics of such arrangements and they provide a complete characterization of free conic-line arrangements with nodes, tacnodes, and ordinary triple points. Here our aim is to extend current research by looking at arrangements admitting positive genera curves. Our first choice is to study arrangements of smooth cubic curves and lines in the complex projective plane such that these arrangements admit nodes, $A_{5}$ singular points, and ordinary triple points. This selection is based on a very recent paper by Dimca, Ilardi, Pokora and Sticlaru \cite{IDPS} where, among many things, the authors want to understand configurations of flex points and the associated arrangements of curves. In particular, we want to explain certain aspects revolving around configurations of flex points associated to smooth cubic curves and the arrangement constructed as unions of smooth cubic curves and lines that are tangent to flex points. Let us recall briefly the main results of the present note.

First of all, let $\mathcal{EL} = \{\mathcal{E}_{1}, ..., \mathcal{E}_{k},\ell_{1}, ..., \ell_{d}\}\subset \mathbb{P}^{2}_{\mathbb{C}}$ be an arrangement consisting of $k\geq 1$ smooth cubic curves and $d\geq 1$ lines admitting $n_{2}$ nodes (i.e., $A_{1}$ singularities), $t_{5}$ singular points of type $A_{5}$, and $n_{3}$ ordinary triple points (i.e., $D_{4}$ singularities). Then for such arrangements we have following combinatorial count:
\begin{equation}
\label{count}
9\binom{k}{2} + 3kd + \binom{d}{2} = n_{2} + 3t_{5} + 3n_{3}.
\end{equation}
\begin{proof}
This naive count follows from the classical B\'ezout theorem. The left-hand side is just the number of intersections among all the curves in a given arrangement. The right hand side follows from the fact that a double intersection point costs one intersection, a triple intersection point costs three intersections, and a singular point of type $A_{5}$ costs three intersections, and this follows from the topological classification of plane curve singularities.

\end{proof}
 Obviously the above naive count is very coarse. In the course of the paper, by the weak combinatorics of a given arrangement of smooth cubic curves and lines $\mathcal{EL}$ we mean the vector $(d,k;n_{2},n_{3},t_{5})\in \mathbb{Z}_{\geq 0}^{5}$. In the theory of line arrangements we have certain inequalities involving weak combinatorics, for instance the celebrated Hirzebruch's inequality for line arrangements in $\mathbb{P}^{2}_{\mathbb{C}}$, see \cite{BHH87}. In the setting of our paper, we show the following result.
 \begin{custom}[see Theorem \ref{weakH}]
Let $\mathcal{EL} = \{\mathcal{E}_{1}, ..., \mathcal{E}_{k},\ell_{1}, ..., \ell_{d}\}\subset \mathbb{P}^{2}_{\mathbb{C}}$ be an arrangement consisting of $k\geq 1$ smooth cubic curves and $d\geq 1$ lines such that $3k+d\geq 6$, admitting only $n_{2}$ nodes, $t_{5}$ points of multiplicity $A_{5}$, and $n_{3}$ ordinary triple points. Then we have
$$27k+n_{2}+\frac{3}{4}n_{3}\geq d + 5t_{5}.$$
 \end{custom}
 Our proof uses an orbifold version of the Bogomolov-Miyaoka inequality for log pairs. \\\\
 Next, we focus on the freeness of our cubic-line arrangements. Let us recall basic definitions. Denote by $S := \mathbb{C}[x,y,z]$ the coordinate ring of $\mathbb{P}^{2}_{\mathbb{C}}$ and for a homogeneous polynomial $f \in S$ we denote  by $J_{f}$ the Jacobian ideal associated with $f$, i.e., the ideal generated by the partial derivatives of $f$. 
\begin{definition}
We say that $C \subset \mathbb{P}^{2}_{\mathbb{C}}$ of degree $d$ given by $f \in S_{d}$ is a free curve if the $S$-module $\textrm{Syz}(J_f)$ is minimally generated by $2$ homogeneous syzygies $\{r_1,r_2\}$ of degrees $d_i=\deg r_i$, ordered such that $$1\leq d_1\leq d_2 \quad \text{ and } \quad d_{1}+d_{2}=d-1.$$ 

The multiset $(d_1,d_2)$ is called the exponents of $C$ and $\{r_1,r_2\}$ is said to be a minimal set of generators for the $S$-module $\textrm{Syz}(J_f).$ 
\end{definition}
In the setting of the above definition, the minimal degree of the Jacobian relations among the partial derivatives of $f$ is defined as
$${\rm mdr}(f) := d_{1}.$$

It is somehow complicated to check the freeness of a given curve using the above definition. However, by a result of du Plessis and Wall \cite{duP}, we have the following effective criterion.
\begin{theorem}[du Plessis-Wall]
\label{DP}
A reduced curve $C \subset \mathbb{P}^{2}_{\mathbb{C}}$ given by $f \in S_{d}$ with ${\rm mdr}(f)\leq (d-1)/2$ is free if and only if
\begin{equation}
\label{duPles}
(d-1)^{2} - d_{1}(d-d_{1}-1) = \tau(C),
\end{equation}
where $\tau(C)$ denotes the total Tjurina number of $C$.
\end{theorem}
 First of all, in the light of the above discussion, we can show the following result.
  \begin{custom}[see Theorem \ref{bound}]
Let $\mathcal{EL} = \{\mathcal{E}_{1}, ..., \mathcal{E}_{k},\ell_{1}, ..., \ell_{d}\}\subset \mathbb{P}^{2}_{\mathbb{C}}$ be an arrangement consisting of $k\geq 1$ smooth cubic curves and $d\geq 1$ lines admitting only $n_{2}$ nodes, $t_{5}$ singularities of type $A_{5}$, and $n_{3}$ ordinary triple points. If $\mathcal{EL}$ is free, then $3k+d \leq 9$.
 \end{custom}
Then we focus our efforts on the classification problem.
\begin{problem}
For a fixed $3k+d \in \{4,5,6,7,8,9\}$ does there exist an arrangement $\mathcal{EL}$ consisting of $k\geq 1$ smooth cubic curves and $d\geq 1$ with nodes, singular points of type $A_{5}$, and ordinary triple points that is free? 
\end{problem} 
In the context of the above question, it turns out that if $3k+d \in \{4,5,8\}$, then there is no single example of a free arrangement. Then, for $3k+d \in \{6,7\}$, we can construct examples of free arrangements using the Fermat cubic and its inflection lines. In the case when $3k+d=9$ we are able to extract four admissible weak combinatorics, but we do not know whether one can realize them geometrically as cubic-line arrangements. To sum up this brief discussion, we have prove the following.
\begin{custom}[see Theorem \ref{class}]
Let $\mathcal{EL} = \{\mathcal{E}_{1}, ..., \mathcal{E}_{k},\ell_{1}, ..., \ell_{d}\}\subset \mathbb{P}^{2}_{\mathbb{C}}$ be an arrangement consisting of $k\geq 1$ smooth cubic curves and $d\geq 1$ lines admitting only nodes, $A_{5}$ singular points, and ordinary triple points. If $\mathcal{EL}$ is free, then $3k+d \in \{6,7,9\}$, possibly except the case $3k+d=9$.
\end{custom}

\section{On the weak combinatorics of cubic-line arrangements}
Our approach towards showing a Hirzebruch-type inequality for cubic-line arrangements with nodes, singularities of type $A_5$, and ordinary triple points is based on Langer's variation on the Miyaoka-Yau inequality \cite{Langer} which uses local orbifold Euler numbers $e_{orb}$ of singular points. 

First of all, using Arnold's classification of singularities \cite{arnold}, we present the local normal forms of the aforementioned singularites, namely
\begin{center}
\begin{tabular}{ll}
$A_{1}$ & $: \, x^{2}+y^{2}  = 0$, \\
$A_{5}$ & $: \, x^{2}+y^{6}  = 0$, \\
$D_{4}$ & $: \, y^{2}x + x^{3}  = 0$.
\end{tabular}
\end{center}
Moreover, we will use the standard convention in the theory of curve arrangements that $A_{1}$ singularities are called nodes and $D_{4}$ singularities are called triple intersection points, and we will use this convention freely throughout the paper.

We work with \textbf{log pairs} $(X,D)$, where $X$ is a complex smooth projective surface and $D$ is a boundary divisor -- it is an effective $\mathbb{Q}$-divisor whose coefficients are $\leq 1$ and such that $K_{X}+D$ is $\mathbb{Q}$-Cartier. We say that a log pair $(X,D)$ is effective if $K_{X}+D$ is effective.
\begin{definition}
Let $(X,D)$ be a log pair and let $f: Y \rightarrow X$ be a proper birational morphism from a normal surface $Y$. Write 
$$K_{Y}+D_{Y} = f^{*}(K_{X}+D)$$
with $f_{*}D_{Y} = D$. If the coefficients of $D_{Y}$ are less than or equal to one for every $f: Y \rightarrow X$, then $(X,D)$ is called a log canonical surface.
\end{definition}

Now we need to recall some information about local orbifold numbers $e_{orb}$ that appear in the context of our arrangements. We must point out that the definition of local orbifold Euler numbers is technical and requires a lot of preparation. Due to this reason, we aim to provide only some of their useful properties and for all necessary details we have to refer directly to \cite[pages 361 - 369]{Langer}. First of all, let us provide two comments that will shed some lights on these numbers, namely:
\begin{itemize}
\item local orbifold Euler numbers are analytic in their nature,
\item if $(\mathbb{C}^{2},D)$ is log canonical at $0$ and ${\rm mult}_{0}(D)$ denotes the multiplicity of $D$ in $0$ (i.e., this is a sum of multiplicities of irreducible components $D_{i}$ counted with appropriate multiplicities), then
$$e_{orb}(0, \mathbb{C}^{2},D) \leq (1-{\rm mult}_{0}(D)/2)^{2},$$
which means that $e_{orb}(x,X,D)\leq 1$ for any log canonical pair $(X,D)$.
\end{itemize}
In our setting we look at log pairs $(\mathbb{P}^{2}_{\mathbb{C}}, \alpha D)$, where $D$ is a boundary divisor consisting of $k$ smooth cubic curves and $d$ lines admitting only $A_{1}, D_{4}, A_{5}$ singularities and $\alpha \in [0,1]\cap \mathbb{Q}$. We need also to recall the local orbifold numbers $e_{orb}$ which appear in the context of our arrangements. We must warn the reader that calculating the actual values for $e_{orb}$ is very difficult in most cases. However, if we assume that our curves admit ${\rm ADE}$ singularities or just ordinary intersections, we use \cite[Theorem 8.7, Theorem 9.4.2]{Langer} to get the values of these numbers. Following the mentioned results, we recall that
\begin{itemize}
    \item if $q$ is a node, then $e_{orb}(q,\mathbb{P}^{2}_{\mathbb{C}}, \alpha D)=(1-\alpha)^2$ if $0 \leq \alpha \leq 1$;
    \item if $q$ is a point of type $A_{5}$, then $e_{orb}(q,\mathbb{P}^{2}_{\mathbb{C}}, \alpha D) =\frac{(4-6\alpha)^{2}}{12}$ if $\frac{1}{3} <  \alpha \leq \frac{2}{3}$;
    \item if $q$ is an ordinary triple point, then $e_{orb}(q,\mathbb{P}^{2}_{\mathbb{C}}, \alpha D) \leq \bigg(1 - \frac{3\alpha}{2}\bigg)^2$ if $0 \leq \alpha \leq \frac{2}{3}$.
\end{itemize}
Now we are ready to show our result devoted to weak combinatorics of our cubic-line arrangements.
\begin{theorem}
\label{weakH}
Let $\mathcal{EL} = \{\mathcal{E}_{1}, ..., \mathcal{E}_{k},\ell_{1}, ..., \ell_{d}\}\subset \mathbb{P}^{2}_{\mathbb{C}}$ be an arrangement consisting of $k\geq 1$ smooth cubic curves and $d\geq 1$ lines such that $3k+d\geq 6$. Assume that $\mathcal{EL}$ admits only $n_{2}$ nodes, $t_{5}$ points of type $A_{5}$, and $n_{3}$ triple intersection points. Then we have
$$27k+n_{2}+\frac{3}{4}n_{3}\geq d + 5t_{5}.$$
\end{theorem}
\begin{proof}
We will follow the path shown in, for example, \cite{Pokora2}. 

Let $D = \mathcal{E}_{1} + ... + \mathcal{E}_{k} + \ell_{1} + ... + \ell_{d}$ be a divisor and we set $m := {\rm deg}(\mathcal{EL}) = 3k + d$ with $k\geq 1$ and $d\geq 1$. For our purposes it is enough to work with the pair $\bigg(\mathbb{P}^{2}_{\mathbb{C}}, \frac{1}{2}D\bigg)$, which is log-canonical by \cite[Proposition 2.3]{DimSer}, and this is just one of the possible references regarding this fact, and such a pair is effective since $3k+d\geq 6$. 

We are going to use the inequality from \cite[Section 11.1]{Langer}, namely 
\begin{equation}
\label{logMY}
\sum_{p \in {\rm Sing}(C)}  3\bigg( \frac{1}{2}\bigg(\mu_{p} - 1\bigg) + 1 - e_{orb}\bigg(p,\mathbb{P}^{2}_{\mathbb{C}}, \frac{1}{2} D\bigg) \bigg) \leq \frac{5}{4}m^{2} - \frac{3}{2}m,
\end{equation}
where $\mu_{p}$ is the Milnor number of a singular point $p \in {\rm Sing}(C)$.
Recall that 
\begin{itemize}
    \item if $p$ is an $A_{1}$ singularity, then its local Milnor number is $\mu_{p}=1$;
    \item if $p$ is a $D_{4}$ singularity, then its local Milnor number is $\mu_{p}=4$;
    \item if $p$ is an $A_{5}$ singularity, then its local Milnor number is $\mu_{p}=5$.
\end{itemize}
First of all, we are going to work with the left-hand side of \eqref{logMY}. Observe that
\begin{multline*}
\textbf{L}  := \sum_{p \in {\rm Sing}(C)}  3\bigg( \frac{1}{2}\bigg(\mu_{p} - 1\bigg) + 1 - e_{orb}\bigg(p,\mathbb{P}^{2}_{\mathbb{C}}, \frac{1}{2} D\bigg) \bigg) \geq \\
3n_{2}\bigg(1 - \frac{1}{4}\bigg) 
+ 3n_{3}\bigg(\frac{3}{2} + 1 - \frac{1}{16}\bigg) +3t_{5}\bigg(2 + 1 - \frac{1}{12}\bigg).  
\end{multline*}
After some simple manipulations we get
$$\textbf{L} \geq \frac{9}{4}n_{2} + \frac{117}{16}n_{3}+\frac{35}{4}t_{5}.$$
Now we look at the right-hand side of \eqref{logMY}. Since $$m^{2} = (3k+d)^2 = 9k+d+2n_{2}+6n_{3}+6t_{5}$$ we have
$$\frac{5}{4}m^{2}-\frac{3}{2}m = \frac{5}{4}(9k+d+2n_{2}+6n_{3}+6t_{5})-\frac{3}{2}(3k+d).$$
We plug data into \eqref{logMY} and we get
$$36n_{2}+117n_{3}+140t_{5} \leq 108k - 4d + 40n_{2}+120n_{3}+120t_{5}.$$
After rearranging, we finally obtain
$$27k+n_{2}+\frac{3}{4}n_{3} \geq d + 5t_{5},$$
which completes the proof.
\end{proof}
\begin{remark}
Our Hirzebruch-type inequality is rather tight. If we take the Fermat cubic and $9$ inflectional lines, i.e., lines tangent to the Fermat cubic at inflection points, we obtain an arrangement with  $n_{2}=27$, $n_{3}=3$, and $t_{5}=9$. We have than
$$54 + \frac{9}{4} = 27 + 27 + \frac{3}{4}\cdot 3 \geq  9 + 5\cdot 9 = 54.$$
\end{remark}
\section{Freeness of cubic-line arrangements}
We start with our bound on the degree of free cubic-line arrangements.
\begin{theorem}
 \label{bound}
    Let $\mathcal{EL} = \{\mathcal{E}_{1}, ..., \mathcal{E}_{k},\ell_{1}, ..., \ell_{d}\}\subset \mathbb{P}^{2}_{\mathbb{C}}$ be an arrangement consisting of $k\geq 1$ smooth cubic curves and $d\geq 1$ lines admitting only $n_{2}$ nodes, $t_{5}$ singularities of type $A_{5}$, and $n_{3}$ ordinary triple points. If $\mathcal{EL}$ is free, then $3k+d \leq 9$.
\end{theorem}
\begin{proof}
We are going to use a result due to Dimca and Sernesi \cite[Theorem 2.1]{DimSer} which tells us that in our case of cubic-line arrangements $\mathcal{EL}$ of degree $m=3k+d$ given by $f \in S_{m}$ one has
\begin{equation}
\label{mdr}
{\rm mdr}(f) \geq \alpha_{\mathcal{EL}}\cdot m-2,
\end{equation}
where $\alpha_{\mathcal{EL}}$ is the Arnold exponent of $\mathcal{EL}$. It is worth recalling that the Arnold exponent of a given reduced curve $C \subset \mathbb{P}^{2}_{\mathbb{C}}$ is defined as the minimum over all Arnold exponents $\alpha_{p}$ of singular points $p$ in $C$. In modern language, the Arnold exponents of singular points are nothing else but the log canonical thresholds of singularities. In the case of ${\rm ADE}$-singularities, their log-canonical thresholds are well-known, see for instance \cite[Page 4]{DimSer}, but let us recall these values for singular points that occur in our investigations:
\begin{itemize}
    \item if $p$ is an $A_{1}$ singularity, then $\alpha_{p}=1$;
    \item if $p$ is a $D_{4}$ singularity, then $\alpha_{p}=2/3$;
    \item if $p$ is an $A_{5}$ singularity, then $\alpha_{p}=2/3$.
\end{itemize}
Then it follows that
$$\alpha_{\mathcal{EL}} = {\rm min}\bigg\{1, \frac{2}{3}, \frac{2}{3} \bigg\} = \frac{2}{3}.$$
Recall that by our assumptions $\mathcal{EL}$ is free, which means that exponents $(d_{1},d_{2})$ satisfy the following properties
$$d_{1} + d_{2} = m-1 \quad \text{and} \quad d_{1}\leq d_{2}.$$
It means that 
$$2d_{1} \leq  d_{1} + d_{2} = m-1,$$
and we finally get $d_{1} \leq (m-1)/2$.
\\
Combining the above considerations into one piece, we have
$$\frac{m-1}{2} \geq {\rm mdr}(f) \geq \alpha_{\mathcal{EL}}\cdot m-2 = \frac{2}{3}m-2,$$
which gives us that $3k+d=m\leq 9$, and this completes our proof.
\end{proof}
It means that if $\mathcal{EL}$ is an arrangement of $k\geq 1$ smooth cubic curves and $d\geq 1$ lines admitting nodes, singularities of type $A_{5}$, and ordinary triple points, then ${\rm deg}(\mathcal{EL}) = 3k+d \in \{4,5,6,7,8,9\}$. In the next part of the paper we focus on a naive classification problem, i.e., we would like to check for which degrees $3k+d \in \{4, 5,6,7,8,9\}$ one can find a free cubic-line arrangement.
\begin{remark}
Theorem \ref{bound} explains the disappointment of the authors in \cite[page 5]{IDPS} where an arrangement, denoted there by $C^{'''}$, consisting of one smooth cubic curve and $9$ inflection lines is not free. 
\end{remark}

Now we proceed with our classification of cubic-line arrangements. We will do it case by case, checking each admissible degree. Let's start with the following general remarks.
\begin{remark} If $C$ is a reduced plane curve admitting $n_{2}$ nodes, $n_{3}$ ordinary triple points, and $t_{5}$ singularities of type $A_{5}$, then the total Tjurina number of $C$, i.e., the sum of all Tjurina numbers over all singular points of $C$, is equal to
$$\tau(C) = n_{2} + 4n_{4} + 5t_{5},$$
and this follows from the fact that ${\rm ADE}$-singularities are quasi-homogeneous \cite{arnold}, and hence the total Tjurina number is equal to the total Milnor number of $C$ -- see \cite{reif} for a detailed explanation.
\end{remark}
\begin{remark}
To verify that certain arrangements are free, we use \verb}SINGULAR} programme for symbolic computations \cite{Singular}. In our case, the crucial point is to find the minimal degree of the Jacobian syzygies, which can be done in many ways. In our case, we use the commend \verb}syz(I)}, where $I$ denotes the associated Jacobian ideal, which computes the syzygies among the partial derivatives explicitly. 
\end{remark}

Let us present our degree-wise classification considerations. First, we suppose that $3k+d=4$. In this situation we have one smooth cubic curve and one line. By the combinatorial count, one has
$$3 = n_{2}+3n_{3}+3t_{5},$$
and we have the following list of possible weak combinatorics: $$(n_{2},n_{3},t_{5}) \in \{(3,0,0),(0,0,1)\}.$$ If we have a reduced plane curve $C$ of degree $4$ which is free, then its total Tjurina number has to be equal to $7$. If we have an arrangement with $3$ nodes, then the total Tjurina number is equal to $3$, and if our arrangement has one $A_{5}$ point, then the total Tjurina number is equal to $5$. This shows that for degree $4$ we do not have free arrangements.
\\

Next, we assume that $3k+d=5$. In this scenario we have one smooth cubic curve and two lines. Our combinatorial count gives us that
$$7 = n_{2}+3n_{3}+3t_{5}.$$
We can easily check that we have $6$ possibilities for weak combinatorics, namely
$$(n_{2},n_{3},t_{5}) \in \{(1,0,2),(1,1,1),(1,2,0),(4,0,1),(4,1,0),(7,0,0)\}.$$
Recall that a reduced plane curve $C$ of degree $5$ is free if $d_{1} = 2$ and its total Tjurina number satisfies $\tau(C) = 12$. Computing naively the total Tjurina number for each weak combinatorics presented above we see that this number is less than or equal to $11$. It shows that for degree $5$ we do not have free arrangements.
\\

We investigate now the case $3k+d=6$. Let us consider the arrangement $\mathcal{EL}_{6}=\{\mathcal{E}_{1}, \ell_{1},\ell_{2}, \ell_{3}\}$ in $\mathbb{P}^{2}_{\mathbb{C}}$ which is given by
$$Q(x,y,z) = (x^{3}+y^{3}+z^{3})\cdot(x^{3}+y^{3}).$$
For this arrangement we have $t_{5}=3$ and $n_{3}=1$. Using \verb}Singular} we can check that $d_{1}={\rm mdr}(Q)=2$ since one has
$$x^{2}\cdot \frac{\partial \, Q}{\partial_{y}} -y^{2}\cdot \frac{\partial \, Q}{\partial_{x}} = 0.$$
Using Theorem \ref{DP}, we obtain that
$$19 = 25 - d_{1}(5-d_{1}) = \tau(\mathcal{EL}_{6})=3\cdot 5 + 1\cdot 4 = 19,$$
which means that $\mathcal{EL}_{6}$ is free.
\\

Now we pass to the case $3k+d=7$. Let us consider the arrangement $\mathcal{EL}_{7}=\{\mathcal{E}_{1}, \ell_{1},\ell_{2}, \ell_{3},\ell_{4}\}$ in $\mathbb{P}^{2}_{\mathbb{C}}$ which is given by
$$Q(x,y,z) = (x^{3}+y^{3}+z^{3})\cdot(x^{3}+y^{3})\cdot(y+z).$$
For this arrangement we have $t_{5}=4$, $n_{2}=3$ and $n_{3}=1$. Using \verb}Singular} we can check that $d_{1}={\rm mdr}(Q)=3$. Using Theorem \ref{DP} we obtain
$$27 = 36 - d_{1}(6-d_{1}) = \tau(\mathcal{EL}_{7})=3\cdot 1 + 1\cdot 4 + 4\cdot 5 = 27,$$
which means that $\mathcal{EL}_{7}$ is free.
\\

We consider now the case $3k+d=8$. Recall that by \cite[Theorem 2.9]{DimcaPokoraMaximizing} a reduced curve $C \subset \mathbb{P}^{2}_{\mathbb{C}}$ of degree $8$ given by $f \in S_{8}$ admitting only ${\rm ADE}$-singularities satisfies the condition that $d_{1} = {\rm mdr}(f) \geq 3$, and $C$ is free if and only if it is maximizing, i.e., $\tau(C)=37$ and $d_{1}=3$.
Now we are going to use Theorem \ref{bound} suitably adapted to our scenario, i.e., $C$ is a cubic-line arrangement with nodes, $A_{5}$ singularities, and ordinary triple points, then we have
$$\frac{7}{2}\geq {\rm mdr}(f) \geq \frac{2}{3}\cdot 8 - 2 = \frac{10}{3}.$$
Since ${\rm mdr}(f)$ is an integer, we have a contradiction.
\\

Finally, we suppose that $3k+d=9$. Using the same argument as above, we have
$$4=\frac{8}{2}\geq {\rm mdr}(f) \geq \frac{2}{3}\cdot 9 - 2 = 4,$$
so the only case to consider here is $d_1=4$.
Assuming the freeness of an arrangement, we can use Theorem \ref{DP} obtaining
\begin{equation}
\label{e1}
48=64 - d_{1}(8-d_{1}) = n_{2}+4n_{3}+5t_{5}.
\end{equation}
If the degree of our arrangements is $9$, we have $(k,d) \in \{(1,6),(2,3)\}$. Let us start with the case $(k,d)=(2,3)$.
By the combinatorial count
\begin{equation}
\label{e2}
30 = n_{2}+3n_{3}+3t_{5}.
\end{equation}
Our problem boils down to find nonnegative integer solutions of the following system of equations:
$$\begin{cases}
48 = n_{2}+4n_{3}+5t_{5}\\
30 = n_{2}+3n_{3}+3t_{3}.
		 \end{cases}$$
We have only two possible solutions, namely 
$$(n_{2},n_{3},t_{5}) \in \{(3,0,9),(0,2,8)\}.$$
At this moment we do not know whether these weak combinatorics can be realized over the complex numbers as cubic-line arrangements and we hope to come back to this issue soon.
\\\\
Let us now pass to the case $(k,d)=(1,6)$. By the combinatorial count we have
\begin{equation}
\label{e6}
33 = n_{2}+3n_{3}+3t_{3}.
\end{equation}
Again, our problem boils down to find nonnegative integer solutions of the following system of equations:
$$\begin{cases}
48 = n_{2}+4n_{3}+5t_{5}\\
33 = n_{2}+3n_{3}+3t_{3}.
		 \end{cases}$$
It turns out that we have four solutions, namely   
$$(n_{2},n_{3},t_{5}) \in \{(0,7,4),(3,5,5),(6,3,6),(9,1,7)\}.$$
Now our Hirzebruch inequality comes into play since we can easily check that weak combinatorics $(k,d;n_{2},n_{3},t_{5}) = (1,6;6,3,6)$ and $(k,d;n_{2},n_{3},t_{5}) = (1,6;9,1,7)$ do not satisfy the inequality

On the other hand, we cannot decide whether the first two weak combinatorics can be realized as cubic line arrangements over the complex numbers. We hope to return to this question as soon as possible with more effective methods.

We can summarize our discussion by the following classification result, which is the main result of our note.
\begin{theorem}
\label{class}
Let $\mathcal{EL} = \{\mathcal{E}_{1}, ..., \mathcal{E}_{k},\ell_{1}, ..., \ell_{d}\}\subset \mathbb{P}^{2}_{\mathbb{C}}$ be an arrangement consisting of $k\geq 1$ smooth cubic curves and $d\geq 1$ lines admitting only nodes, $A_{5}$ singular points, and ordinary triple points. If $\mathcal{EL}$ is free, then $3k+d \in \{6,7,9\}$, possibly except the case $3k+d=9$.
\end{theorem}
\section*{Conflict of Interests}
The author declares that there is no conflict of interest regarding the publication of this paper.
\section*{Acknowledgment}
This note is part of the author's Master's thesis, written under the supervision of Piotr Pokora. Moreover, the author is partially supported by The Excellent Small Working Groups Programme DNWZ.711/IDUB/ESWG/2023/01/00002 at the Pedagogical University of Cracow.

\vskip 0.5 cm

\bigskip
Przemys\l aw Talar,
Department of Mathematics,
University of the National Education Commission Krakow,
Podchor\c a\.zych 2,
PL-30-084 Krak\'ow, Poland. \\
\nopagebreak
\textit{E-mail address:} \texttt{p.talar367@gmail.com}
\bigskip
\end{document}